\documentclass[12pt]{amsart}

\usepackage{amssymb,url,mathrsfs,euscript}

\newtheorem{teo}{Theorem}[section]
\newtheorem{lema}{Lemma}[section]
\newtheorem{pro}{Proposition}[section]

\newtheorem{rmk}{Remark}[section]\newtheorem*{rmk*}{Remark}

\newtheorem{coro}{Corollary}[section]

\theoremstyle{remark}     
\newtheorem*{ex*}{Example}
\newtheorem*{exs*}{Examples}

\def\sideremark#1{\ifvmode\leavevmode\fi\vadjust{\vbox to0pt{\vss
\hbox to 0pt{\hskip\hsize\hskip1em%
\vbox{\hsize2cm\tiny\raggedright\pretolerance10000%
\noindent #1\hfill}\hss}\vbox to8pt{\vfil}\vss}}}%

\swapnumbers              
\theoremstyle{plain}      
\newtheorem{lemma}[rmk]{Lemma}

\newtheorem{theorem}[rmk]{Theorem}

\theoremstyle{definition} 

\newcommand{\bt}{\begin{theorem}}\newcommand{\et}{\end{theorem}}
\newcommand{\bl}{\begin{lemma}}\newcommand{\el}{\end{lemma}}
\newcommand{\bp}{\begin{proof}}\newcommand{\ep}{\end{proof}}

\newcommand{\be}{\begin{equation}}\newcommand{\ee}{\end{equation}}
\newcommand{\bdm}{\begin{displaymath}}
\newcommand{\edm}{\end{displaymath}}

\numberwithin{equation}{section}

\def \V{{V}}

\def \V{\mathcal{V}}

\renewcommand{\o}{\omega}

\renewcommand{\leq}{\leqslant}
\DeclareMathOperator{\s}{\textsl{s}}

\newcommand{\grad}{\textsl{grad\,}}

\newcommand{\tr}{\textsl{tr}\,}

\newcommand{\bb}{\mathbb}
\newcommand{\ba}{\begin{array}}\newcommand{\ea}{\end{array}}

\renewcommand{\&}{{\footnotesize \&}}

\begin{document}

\title[]{On the uniqueness of almost-K\"ahler structures}
\thanks{Partially supported by {\sc Gnsaga} of INdAM, {\sc Prin} \oldstylenums{2007} of {\sc Miur} (Italy), and the Royal Society of New Zealand, Marsden grant no. 06-UOA-029}
\date{\today}
\author[A.J. di Scala]{Antonio J. di Scala}
\address[AJS]{Dipartimento di Matematica, Politecnico di Torino, Corso Duca degli Abruzzi 24, I-10129, Torino}
\email{antonio.discala@polito.it}
\author[P.-A. Nagy]{Paul-Andi Nagy} 
\address[PAN]{Department of Mathematics, University of Auckland, 
38 Princes St., Auckland, New Zealand}
\email{nagy@math.auckland.ac.nz}

\frenchspacing

\begin{abstract}
We show uniqueness up to sign of positive, orthogonal almost-K\"ahler structures on any non-scalar flat K\"ahler-Einstein surface. 
\end{abstract}

\date{\today}
\maketitle

\section{Introduction}
Given a Riemannian four manifold $(M,g)$ an orthogonal almost-K\"ahler structure is an orthogonal almost complex structure $J$ such that $g(J \cdot, \cdot)$ is a closed two form. It seems to be quite difficult to determine, in general, the obstructions 
on the metric $g$ to the existence of such a structure. While integrability results are available for Einstein, anti-self-dual metrics \cite{arm}, familiar examples such as the Kodaira-Thurston manifold support (see e.g. \cite{surv})
a full circle of almost-K\"ahler structures on the same orientation. 

In this note we explore the uniqueness question of almost K\"ahler structures orthogonal with respect to a K\"ahler metric and compatible with the positive orientation.  
\begin{teo} \label{main}
Let $(M,g,J)$ be a connected K\"ahler-Einstein surface, positively oriented by $J$. Then any positive, orthogonal almost K\"ahler structure $J^{\prime}$ is integrable; in particular 
$J^{\prime}=\pm J$ when $g$ has non-vanishing scalar curvature.
\end{teo}
The proof of Theorem \ref{main} is build around the observation that the angle function between the structures $J$ and $J^{\prime}$ is isoparametric in the sense of Cartan. 
The conclusion follows by investigation of the eigenvalue structure of the second fundamental form of its level sets. 

In particular we obtain the following, that answers a question posed in \cite{surv}.
\begin{coro} \label{sym}Any orthogonal almost-K\"ahler structure on (some open piece of) $\bb{C}\bb{H}^1 \times \bb{C}\bb{H}^1$ or $\bb{C}\bb{H}^2$ equipped with their canonical metrics is integrable. 
\end{coro} 
\begin{proof}
Both spaces are equipped with K\"ahler Einstein metrics of negative scalar curvature. In the case of $\bb{CH}^2$, an orthogonal almost-K\"ahler structure is either compatible with the negative orientation, where integrability follows from 
\cite{arm} since the metric is self-dual, or with the positive orientation, where the same conclusion is obtained by Theorem \ref{main}. For $\bb{CH}^1 \times \bb{CH}^1$ which admits a K\"ahler structure on either orientation it 
suffices to apply Theorem \ref{main} twice.
\end{proof}
\section{Almost K\"ahler structures on K\"ahler surfaces}
Let $(M,g,J)$ be a K\"ahler surface, positively oriented by $J$ and admitting a positive orthogonal almost-K\"ahler structure $J^{\prime}$. Recall that the bundle $\Lambda^{+}M$ of self-dual forms on $M$ splits as
\begin{equation} \label{split1} 
\Lambda^{+}M=\langle \o_J \rangle \oplus \lambda^2_JM
\end{equation}
where $\lambda^2_JM=\{\alpha \in \Lambda^2M : \alpha (J \cdot ,J \cdot )=-\alpha\}$. The K\"ahler form $\o_{J^{\prime}}=g(J^{\prime} \cdot, \cdot)$ of $J^{\prime}$ is self-dual hence 
$$\o_{J^{\prime}}=x\o_J+\Psi$$
along \eqref{split1}, where $\o_J=g(J \cdot ,\cdot)$ is the K\"ahler form of $J$ and $\Psi$ belongs to $\lambda^2_JM$.The angle map $x:M \to \mathbb{R}$ is explicitly given by $x=-\frac{1}{4}\langle J^{\prime},J\rangle$.

Since $\o_{J^{\prime}}$ is a closed form in $\Lambda^{+}M$ it is in 
particular harmonic, thus $\Delta x=0$ because the splitting \eqref{split1} is parallel w.r.t the Levi-Civita connection of $g$ and hence preserved by the Laplace operator. At this stage it is straightforward to treat the compact case.
\begin{pro} \label{co}
If $M$ is compact then either $J^{\prime}=\pm J$ or the scalar curvature $\s$ of the metric $g$ vanishes.
\end{pro}
\begin{proof}
Since $M$ is compact, $x$ must be constant hence $\Psi$ either vanishes or it has constant, non-zero norm. As it is well known, the latter case forces $\s=0$.
\end{proof}
Let now $D$ be the open set where $\Psi$ is non zero, to be assumed not empty in what follows. After re-normalisation of $\Psi$ we can write $J^{\prime}=xJ+yI$ where $I$ is a local gauge for $J$ and $x^2+y^2=1$. That is $I$ is an orthogonal 
almost complex structure such that $IJ+JI=0$. We obtain a local gauge $I^{\prime}$ for $J^{\prime}$ given by 
\begin{equation} \label{gauge-new}
I^{\prime}=-yJ+xI, \ K^{\prime}=I^{\prime}J^{\prime}=K,
\end{equation}
where $K=IJ$. Let $\nabla$ denote the Levi-Civita connection of $g$.

Because $J$ is K\"ahler we have $\nabla I=b \otimes K$ for some local $1$-form $b$ on $M$ such that $db=\rho^J$, where $\rho^J=g(Ric \circ J \cdot, \cdot)$ is the Ricci form of $(g,J)$.
Here the Riemann curvature tensor is defined by $R(X,Y)=-\nabla^2_{X,Y}+\nabla^2_{Y,X}$ for all $X,Y$ in $TM$ and $Ric$ is the Ricci contraction. 

The local connection $1$-forms $a^{\prime}$ and $b^{\prime}$ of the gauge $I^{\prime}$ are determined from the almost K\"ahler condition on $(g,J^{\prime})$, that is $d\o_{J^{\prime}}=0$, by 
\begin{equation} \label{lcf}
\nabla J^{\prime}=a^{\prime} \otimes I^{\prime}-J^{\prime}a^{\prime} \otimes K^{\prime}, \ \nabla I^{\prime}=-a^{\prime} \otimes J^{\prime}+b^{\prime}\otimes K^{\prime}.
\end{equation}
The action of $J^{\prime}$ on $1$-forms on $M$ is defined by $J^{\prime}\alpha=\alpha(J^{\prime} \cdot)$.  
Without loss of generality we write $x=\cos \theta, y=\sin \theta$ for some local function 
$\theta$ on $M$.
\begin{lema} \label{prel}The following hold on $D$:
\begin{itemize}
\item[(i)] $a^{\prime}=d\theta$ and $J^{\prime}a^{\prime}=-\sin \theta b$;
\item[(ii)] $b^{\prime}=\cos \theta b$;
\item[(iii)] $\vert \grad \theta \vert^2=-\s \frac{\sin^2 \theta }{4}$.
\end{itemize}
\end{lema}
\begin{proof}
Let us record first the following inversion formulae
\begin{equation} \label{inv}
J=xJ^{\prime}-yI^{\prime}, \ I=yJ^{\prime}+xI^{\prime}.
\end{equation}
(i) We compute 
\begin{equation*}
\nabla J^{\prime}=dx \otimes J+yb \otimes K+dy \otimes I=(xdy-ydx) \otimes I^{\prime}+(yb) \otimes K^{\prime}
\end{equation*}
after using \eqref{inv}. Therefore $a^{\prime}=xdy-ydx=d\theta$ and $yb=-J^{\prime}a^{\prime}$ from the almost K\"ahler condition.\\ 
(ii) As in (i), we differentiate in \eqref{gauge-new} to find 
\begin{equation*}
\nabla I^{\prime}=-dy \otimes J+xb \otimes K+dx \otimes I=(ydx-xdy) \otimes J^{\prime}+xb \otimes K^{\prime}
\end{equation*}
after using \eqref{inv} and $x^2+y^2=1$. The claim is proved by comparison with \eqref{lcf}. \\
(iii) Because $(g,J)$ is K\"ahler the positive Weyl tensor is given by 
$$W^{+}=\biggl ( \begin{array}{lr}
\frac{s}{6} & 0\\
0 & -\frac{s}{12}\\
\end{array} \biggr ) $$
with respect to the decomposition \eqref{split1}. In particular, $W^{+}\o_{J^{\prime}}=\frac{\s}{6}(x\o_J-\frac{y}{2}\o_I)$, making that the conformal scalar curvature $\kappa=3\langle W^{+}\o_{J^{\prime}}, \o_{J^{\prime}}\rangle$ of $(g,J^{\prime})$ 
is given by $\kappa=(x^2-\frac{y^2}{2})\s$. The well known (see \cite{surv}) relation $\frac{\kappa-\s}{3}=\frac{1}{4} \vert \nabla J^{\prime} \vert^2$ combined with \eqref{lcf} and (i) yields the claim.
\end{proof}
In particular, Theorem \ref{main} folllows immediately when $\s>0$. Therefore, we assume from now on that the metric $g$ is Einstein, that is $\rho^J=\frac{\s}{4}\o_J$ and re-normalise the scalar curvature to $\s=-4$. We may assume, w.l.og. that 
$\sin \theta>0$ so that $\vert \grad \theta \vert=\sin \theta$ by Lemma \ref{prel}, (iii). The unit vector field $\xi=\frac{\grad \theta }{\vert \grad \theta \vert}$ is then totally geodesic henceforth normal 
to the co-dimension one Riemannian foliation induced by $\V=\ker d\theta$. 

Let $S$ in $S^2\V$ be given by $\langle SV,W \rangle=(\nabla_V \xi)W $ for all $V,W$ in $\V$. It describes the second fundamental form of the distribution $\V$. The unit vector field $\xi_1=J^{\prime}\xi$ in $\V$ has dual one form $\xi^1$ subject to 
\begin{equation} \label{str-1}
d \xi^1=-\o_J.
\end{equation}
since $\xi^1=b$ by (i) in Lemma \ref{prel}.
\begin{lema} \label{l3} The following hold:
\begin{itemize}
\item[(i)] $\tr S=-2\cos \theta$;
\item[(ii)] $\tr S^2=2 \cos^2 \theta-1$;
\item[(iii)] $S(\xi_1)=-(\cos \theta) \xi_1$.
\end{itemize}
\end{lema}
\begin{proof}
(i) We have $\sum \limits_{i=1}^3\nabla_{e_i}(dx)e_i+\nabla_{\xi}(dx)\xi=0$ for an arbitrary local orthonormal frame in $\V$ since the angle function is harmonic. Since 
\begin{equation} \label{comp}
dx=-(\sin^2 \theta) \xi
\end{equation}
we have $\langle \nabla_{\xi}(\sin^2 \theta \xi), \xi \rangle=2 \sin^2 \theta \cos \theta$, and its now easy to conclude. \\
(ii) Because $x$ is harmonic, we have $\Delta(dx)=0$ and hence $(\nabla^{\star}\nabla)dx=dx$ by the Bochner formula. Taking the scalar product with $dx$ we get further 
$$ \frac{1}{2}\Delta \vert dx \vert^2=\langle \nabla^{\star} \nabla(dx), dx \rangle-\vert \nabla dx \vert^2=\vert dx \vert^2-\vert \nabla(d x) \vert^2.
$$
By \eqref{comp} we have $\vert dx \vert=\sin^2 \theta=1-x^2$. Using again that $x$ is harmonic we obtain $\frac{1}{2}\Delta \vert dx \vert^2=2\vert dx \vert^2-6x^2 \vert dx \vert^2$ thus 
$\vert \nabla(d x) \vert^2=(6x^2-1) \vert dx \vert^2$. After expressing $dx$ in terms of $\xi$, which is of unit length, we get $\vert \nabla(d x) \vert^2=\sin^4\theta \vert \nabla \xi \vert^2+4 \sin^4 \theta \cos^2 \theta $ and the claim follows.\\ 
(iii) By \eqref{str-1} we have $d\xi^1=-\o_{J}$ hence $\langle \nabla_{X}J^{\prime}\xi,Y\rangle-\langle \nabla_Y(J^{\prime}\xi),X\rangle=-\langle JX,Y\rangle$ for all $X,Y$ in $TM$. Taking 
$Y=J^{\prime}\xi$ we get $\nabla_{J^{\prime}\xi}(J^{\prime}\xi)=-(JJ^{\prime})\xi$.  Therefore 
$$J^{\prime} \nabla_{\xi_1}\xi=-(\nabla_{\xi_1}J^{\prime})\xi-(JJ^{\prime})\xi. $$ However, 
by means of \eqref{gauge-new} and the definition of $J^{\prime}$ we get $(\nabla_{\xi_1}J^{\prime})\xi=(\sin \theta) K\xi$ while \eqref{inv} yields $(J J^{\prime})\xi=-(\cos \theta )\xi-(\sin \theta) K \xi$ and the claim follows.
\end{proof}

$\\$
{\bf{Proof of Theorem \ref{main}:}}\\
On the orthogonal complement $\V_0$ of $\xi_1$ in $\V$ we must have $\tr_{\V_0}S^2=\cos^2 \theta-1 \leq 0$ by (ii) and (iii) in Lemma \ref{l3}. Since $S$ is symmetric and preserves $\V_0$ it follows that 
$\cos^2 \theta=0$ which is contradictory to having $D$ non-empty.   


\begin{thebibliography}{99}
\bibitem{surv}
V.Apostolov, T.Dr\u{a}ghici, \textit{The curvature and the integrability of almost-K\"ahler manifolds: a survey.}  Symplectic and contact topology: interactions and perspectives (Toronto, ON/Montr\'eal, QC, 2001),  25-53,
Fields Inst. Commun., {\bf{35}}, Amer. Math. Soc., Providence, RI, 2003. 
\bibitem{arm}
J.Armstrong, \textit{An ansatz for almost-K\"ahler, Einstein $4$-manifolds}, 
J. Reine Angew. Math. {\bf{542}} (2002), 53-84. 
\end{thebibliography}
\end{document}